\documentclass[psamsfont,a4paper,11pt]{amsart}

\usepackage[english]{babel}   

\usepackage{setspace}     
\usepackage{enumerate} 	
\usepackage{afterpage} 	
\usepackage[usenames,dvipsnames]{color}
\usepackage{graphicx}	

\usepackage{amsmath} 	
\usepackage{amssymb}	
\usepackage{mathrsfs} 	
\usepackage{amsfonts} 	
\usepackage{latexsym} 	
\usepackage[latin1]{inputenc} 

\usepackage{amsthm} 	
\usepackage{stackrel}       
\usepackage{amscd} 	

\usepackage{tabularx}	
\usepackage{longtable}	

\usepackage{verbatim} 
\usepackage{blindtext} 

\usepackage{amssymb,latexsym,amsmath}    
\usepackage{pstricks,multido,pst-plot}		 
\usepackage{multicol,fancyheadings}			   
\usepackage{float}                       
\usepackage[vcentermath,enableskew]{youngtab}

\usepackage{subfig}

\allowdisplaybreaks

\def\a{\alpha} 
 
\def\b{\beta} 
 
\def\e{\epsilon} 
 
\def\m{\mu}

\def\t{\tau} 
 
\def\p{\mathfrak{p}}

\def\pp{\mathfrak{P}}

\def\cal{\mathcal}

\def\O{{\cal O}}

\newcommand{\Z}{\mathbb{Z}} 
\newcommand{\Q}{\mathbb{Q}}

\def\Li{{\mathrm{Li}}}
\def\exp{{\mathrm{exp}}}
\def\odot{{\ddot{\mathrm{o}}}}

\newtheorem{Thm}{Theorem}[]		
\newtheorem{Lemma}[Thm]{Lemma}

\newtheorem*{thm}{Theorem}	

\newtheorem*{cor}{Corollary}

\theoremstyle{definition}

\theoremstyle{remark}
 
\newtheorem*{rmk}{Remark}

\newtheorem*{Examples}{Examples}

\newtheorem{ind}[]{{\rm\it Indice}}

\usepackage{multicol}

\usepackage{tikz}
\usetikzlibrary{arrows}

\usepackage{paralist}
\usepackage{cite}

\title{A new formula for Chebotarev Densities}

\author[Locus]{Madeline Locus Dawsey*}
\address{Department of Mathematics and Computer Science,
Emory University, Atlanta, GA 30322}
\email{madeline.locus@emory.edu}

\begin{document}

\thanks{*This author was previously known as Madeline Locus.}

\begin{abstract}
We give a new formula for the Chebotarev densities of Frobenius elements in Galois groups.  This formula is given in terms of smallest prime factors $p_{\mathrm{min}}(n)$ of integers $n\geq2$.  More precisely, let $C$ be a conjugacy class of the Galois group of some finite Galois extension $K$ of $\Q$.  Then we prove that
\begin{equation*}
-\lim_{X\rightarrow\infty}\sum_{\substack{2\leq n\leq X\\[1pt]\left[\frac{K/\Q}{p_{\mathrm{min}}(n)}\right]=C}}\frac{\m(n)}{n}=\frac{\#C}{\#G}.
\end{equation*}
This theorem is a generalization of a result of Alladi from 1977 that asserts that largest prime divisors $p_{\mathrm{max}}(n)$ are equidistributed in arithmetic progressions modulo an integer $k$, which occurs when $K$ is a cyclotomic field $\Q(\zeta_k)$.
\end{abstract}

\maketitle

\section{Introduction and Statement of Results}\label{1}

It is well-known that
\begin{equation*}
\lim_{s\rightarrow1^+}\zeta(s)^{-1}=\lim\limits_{X\rightarrow\infty}\sum_{n=1}^X\frac{\m(n)}{n}=0,
\end{equation*}
where $\zeta(s)=\sum_{n\geq1}n^{-s}$ is the Riemann zeta function and $\m(n)$ is the M$\odot$bius function defined by
\begin{equation*}
\m(n):=\left\{
		\begin{array}{ll}
			1 & \mbox{if } n=1, \\
			0 & \mbox{if } p^2\mid n\text{ for some prime }p, \\
			(-1)^r & \mbox{if } n=p_1\cdots p_r\text{ where the }p_i\text{ are distinct primes}.
		\end{array}
	\right.
\end{equation*}
This follows from the fact that $\zeta(s)$ has a pole at $\mathrm{Re}(s)=1$.  In particular, we observe by reordering that
\begin{equation*}
-\lim\limits_{X\rightarrow\infty}\sum_{n=2}^X\frac{\m(n)}{n}=1.
\end{equation*}

As this work will show, this can be interpreted as the statement that 100\% of integers $n\geq2$ are divisible by a prime.  To make sense of this, let $p_{\mathrm{min}}(n)$ (resp. $p_{\mathrm{max}}(n)$) denote the smallest (resp. largest) prime divisor of $n$.  Alladi shows in \cite{A} that if $\gcd(\ell,k)=1$, then
\begin{equation}\label{Alladi}
-\sum_{\substack{n\geq2\\[1pt]p_{\mathrm{min}}(n)\equiv\ell\pmod{k}}}\frac{\m(n)}{n}=\frac{1}{\varphi(k)},
\end{equation}
which, by Dirichlet's theorem on primes in arithmetic progressions and a certain duality principle stated in Section \ref{2}, implies that \emph{largest} prime divisors are equidistributed in arithmetic progressions modulo $k$.  Here we generalize this.

For aesthetic purposes, we define
\begin{equation}
\m^*(n):=-\m(n).
\end{equation}

In order to state our results, suppose that $K$ is a Galois extension of $\Q$.  If $p$ is an unramified prime and $\p\subseteq\O_K$ is a prime ideal lying above $p$, then let $\left[\frac{K/\Q}{\p}\right]$ denote the Artin symbol (see Section \ref{2.2}).  For convenience, we let
\begin{equation*}
\left[\frac{K/\Q}{p}\right]=\left\{\left[\frac{K/\Q}{\p}\right]\,:\,\p\subseteq\O_K\text{ is a prime ideal lying above }p\right\}.
\end{equation*}
Therefore, it is well-known that $\left[\frac{K/\Q}{p}\right]=C$, where $C\subset\mathrm{Gal}(K/\Q)$ is a conjugacy class.

\begin{Thm}\label{main}
Let $K$ be a finite Galois extension of $\Q$ with Galois group $G=\mathrm{Gal}(K/\Q)$, and let $C\subset G$ be a conjugacy class.  Then we have that
\begin{equation*}
\lim_{X\rightarrow\infty}\sum_{\substack{2\leq n\leq X \\[1pt] \left[\frac{K/\Q}{p_{\mathrm{min}}(n)}\right]=C}}
\frac{\m^*(n)}{n}=\frac{\#C}{\#G}.
\end{equation*}
\end{Thm}

\begin{rmk}
The convergence of the sum in Theorem \ref{main} is conditional, and the proof of Theorem \ref{main} in Section \ref{3} gives an explicit convergence rate (see equation \eqref{convrate}).
\end{rmk}

\begin{rmk}
We can view the set $\left\{-\frac{\m(2)}{2},-\frac{\m(3)}{3},-\frac{\m(5)}{5},-\frac{\m(6)}{6},\dots\right\}=\left\{\frac{\m^*(2)}{2},\frac{\m^*(3)}{3},\frac{\m^*(5)}{5},\frac{\m^*(6)}{6},\dots\right\}$ as a ``signed probability measure" which can be used to calculate Chebotarev densities via minimal prime divisors of squarefree integers.
\end{rmk}

\vspace{.1cm}

\begin{Examples}
\item
\begin{enumerate}[\hspace{.3cm}a)]
\item Alladi's theorem (\ref{Alladi}) is a special case of Theorem \ref{main} where one chooses $K$ to be a cyclotomic field, i.e. $K=\Q\left(\zeta_k\right)$, where $\zeta_k$ is a primitive $k$th root of unity.\\
\item Let $f(x)=x^4+x+1$.  Then $f(x)$ has Galois group $\mathrm{Gal}(f)=S_4$, so in particular $\#\mathrm{Gal}(f)=24$.  Let $K$ be the splitting field of $f$, and define the set
\begin{equation*}
S:=\{p\text{ prime}\,:\,p\text{ is unramified in }K\text{ and }f\text{ has no roots in }\Z/p\Z\}.
\end{equation*}
For primes $p\in S$, the reduction of $f$ modulo $p$ is either an irreducible quartic, which corresponds to the conjugacy class in $\mathrm{Gal}(f)=S_4$ consisting of a 4-cycle (this conjugacy class has six elements), or a product of two irreducible quadratics, which corresponds to the conjugacy class in $S_4$ consisting of products of two transpositions (this conjugacy class has three elements).  Then the probability of an irreducible quartic contributes $\frac{6}{24}$ to the sum, and the probability of a product of irreducible quadratics contributes $\frac{3}{24}$ to the sum, so the theorem gives
\begin{equation*}
\sum_{\substack{n\geq2\\[1pt]p_{\mathrm{min}}(n)\in S}}\frac{\m^*(n)}{n}=\frac{3}{8}=0.375.
\end{equation*}
Now, define the set
\begin{equation*}
S':=\{p\text{ prime}\,:\,p\text{ is unramified in }K\text{ and }f\text{ has exactly one root in }\Z/p\Z\}.
\end{equation*}
For primes $p\in S'$, the reduction of $f$ modulo $p$ is a product of a linear factor and an irreducible cubic, which corresponds to the conjugacy class in $S_4$ consisting of a 3-cycle (this conjugacy class has eight elements).  Then the theorem gives
\begin{equation*}
\sum_{\substack{n\geq2\\[1pt]p_{\mathrm{min}}(n)\in S'}}\frac{\m^*(n)}{n}=\frac{1}{3}.
\end{equation*}
The table below gives the actual values of the sums
\begin{equation*}
\sum_{\substack{2\leq n\leq X\\[1pt]p_{\mathrm{min}}(n)\in S}}\frac{\m^*(n)}{n}\hspace{.5cm}\text{and}\hspace{.5cm}\sum_{\substack{2\leq n\leq X\\[1pt]p_{\mathrm{min}}(n)\in S'}}\frac{\m^*(n)}{n}
\end{equation*}
for increasing values of $X$.\\
\begin{center}
\begin{tabular}{ |c|c|c|c| } 
 \hline
$X$ & $f\pmod{p}$ has no roots & $f\pmod{p}$ has 1 root \\ 
 \hline
 \hline
 20,000 & 0.3730 & 0.3342\\
 \hline
 40,000 & 0.3741 & 0.3328\\
 \hline
60,000 & 0.3738 & 0.3337\\ 
 \hline
80,000 & 0.3735 & 0.3330\\ 
 \hline
100,000 & 0.3734 & 0.3338\\
 \hline
\end{tabular}
\end{center}
\vspace{.3cm}
\item Theorem \ref{main} holds for suitable sets of primes $S$ with Dirichlet density.
\end{enumerate}
\end{Examples}

\vspace{.4cm}

To prove Theorem \ref{main}, we need the following theorem which is a statement about the \emph{largest} prime divisors of integers.

\begin{Thm}\label{lpd}
Assume the notation and hypotheses from Theorem \ref{main}.  Then we have that
\begin{equation}\label{lpdeq}
\sum_{\substack{2\leq n\leq X \\[1pt] \left[\frac{K/\Q}{p_{\mathrm{max}}(n)}\right]=C}}
1=\frac{\#C}{\#G}\cdot X+O\left(X\,\mathrm{exp}\left\{-k(\log X)^{1/3}\right\}\right),
\end{equation}
where $k$ is a positive constant.
\end{Thm}

In Section \ref{2.1}, we give some results which will help bound error terms in the proofs of the above theorems, and we give the duality between largest and smallest prime divisors introduced by Alladi \cite{A}.  In Section \ref{2.2}, we state and explain the Chebotarev Density Theorem \cite{C}.  In Section \ref{3}, we prove Theorems \ref{main} and \ref{lpd}.

\vspace{.4cm}

\section*{Acknowledgements}
The author would like to thank Krishnaswami Alladi for sharing his inspiring results and ideas, Ken Ono for his invaluable advice and guidance, and Jesse Thorner for his helpful comments.  The author would also like to thank the reviewers for taking the time to give incredibly detailed comments; among other things, the reviewers suggested a clearer proof of Lemma \ref{tetris} and a better explanation of equation \eqref{dickman}.

\vspace{.4cm}

\section{Preliminaries}\label{2}

\subsection{Error Terms and Duality}\label{2.1}

We require several tools in order to prove Theorems \ref{main} and \ref{lpd}.  First, we define the function
\begin{equation*}
\Psi(X,Y):=\sum_{\substack{n\leq X \\[1pt]p_{\mathrm{max}}(n)\leq Y}}1,
\end{equation*}
which counts the number of integers $n\leq X$ with largest prime divisor $p_{\mathrm{max}}(n)\leq Y$.  Let $\mathcal{S}(X,Y)$ denote the set of such integers $n\leq X$ with $p_{\mathrm{max}}(n)\leq Y$.  Then clearly $|\mathcal{S}(X,Y)|=\Psi(X,Y)$.

We now state a theorem of Hildebrand \cite{H} which improves an asymptotic bound for $\Psi(X,Y)$ given by de Bruijn \cite{B}.  We must first define the \emph{Dickman function} $\rho(\b)$ as the continuous solution of the system
\begin{eqnarray*}
\rho(\b)&=&1\hspace{.3 cm}\text{for }0\leq\b\leq1,\\
-\b\rho'(\b)&=&\rho(\b-1)\hspace{.3cm}\text{for }\b>1.
\end{eqnarray*}

\begin{thm}[Hildebrand]
We have that
\begin{equation*}
\Psi(X,Y)=X\rho(\b)\left(1+O_\e\left(\frac{\b\log(\b+1)}{\log X}\right)\right)
\end{equation*}
uniformly in the range $X\geq3,1\leq\b\leq\log X/(\log\log X)^{5/3+\e},$ for any fixed $\e>0$.
\end{thm}

The following unpublished theorem of Maier \cite{M} can be recognized as a corollary of Hildebrand's Theorem, and this corollary will be sufficient to prove the theorems in this paper.

\begin{cor}
If $\b=\frac{\log X}{\log Y}$, then for $X$ sufficiently large (where $\b$ varies with $X$) we have that
\begin{equation*}
\Psi(X,Y)\sim X\rho(\b)
\end{equation*}
uniformly in the range $1\leq\b\leq(\log X)^{1-\e}$ for any fixed $\e>0$.
\end{cor}

It turns out that for $1\leq\b\leq(\log X)^{1-\e}$, we have
\begin{equation}\label{dickman}
\Psi(X,Y)=O_\e\left(X\,\mathrm{exp}\left\{-\b\log\b/2\right\}\right).
\end{equation}
The $O$-constant depends on $\e$, and we will later choose $\e=2/3$ in the proof of Theorem \ref{lpd}.  To obtain $\eqref{dickman}$, we require an upper bound for $\rho(\b)$ as recommended by one of the reviewers.  Norton \cite[Lemma 4.7]{N} gives the following bound,
\begin{equation}\label{dickman2}
\rho(\b)\leq\frac{1}{\Gamma(\b+1)}.
\end{equation}
Applying Stirling's formula to \eqref{dickman2}, we see that
\begin{equation*}
\rho(\b)\sim\frac{1}{\sqrt{2\pi\b}}\,\exp\left\{-\b\log\b/e\right\}.
\end{equation*}
From this asymptotic estimate and the above corollary, it is straightforward to see that \eqref{dickman} holds.

These estimates will be useful in bounding error terms in the proof of Theorem \ref{main}.  The following lemma will also be useful in obtaining estimates.

\begin{Lemma}\label{tetris}
For $a\leq X$ and $\mathcal{S}(X,Y)$ defined as above, we have that
\begin{equation*}
\int_a^X\left(\sum_{n\in \mathcal{S}(X/t,t)}1\right)dt=\sum_{\substack{1\leq n\leq X/a\\[1pt]p_{\mathrm{max}}(n)\leq X/n}}\int_{\max\big(p_{\mathrm{max}}(n),a\big)}^{X/n}dt.
\end{equation*}
\end{Lemma}

\begin{proof}
Looking at the (Stieltjes) integral on the left hand side, a given integer $n$ occurs whenever $t$ is in the range $t\geq a$, $t\geq p_{\mathrm{max}}(n)$, and $t\leq X/n$.  Therefore the integer $n$ contributes to the integral the length of the interval from $\max\big(p_{\mathrm{max}}(n),a\big)$ to $X/n$, provided that $\max\big(p_{\mathrm{max}}(n),a\big)\leq X/n$.  This is precisely the contribution of $n$ to the sum on the right hand side.
\end{proof}

In addition to error bounds, the proof of Theorem \ref{main} requires the following beautiful result of Alladi on the duality between largest and smallest prime factors.

\begin{thm}[Alladi \cite{A}]
If $f$ is a function defined on integers with $f(1)=0$, then
\begin{eqnarray}
\label{Alladi1}\sum_{d\vert n}\m(d)f\big(p_{\mathrm{max}}(d)\big)&=&-f\big(p_{\mathrm{min}}(n)\big),\\
\sum_{d\vert n}\m(d)f\big(p_{\mathrm{min}}(d)\big)&=&-f\big(p_{\mathrm{max}}(n)\big).\label{Alladi2}
\end{eqnarray}
\end{thm}

This theorem provides the connection between Theorems \ref{main} and \ref{lpd}.

\vspace{.3cm}

\subsection{The Chebotarev Density Theorem}\label{2.2}

Our main result is closely related to the Chebotarev Density Theorem, which we carefully state here.  We must first give all of the machinery required to define the Artin symbol.  Let $L/K$ be a finite Galois extension of number fields, and let $\O_L$ and $\O_K$ be the corresponding rings of integers.  Let $\p$ be any nonzero prime (maximal) ideal in $\O_K$.  Then the ideal generated in $\O_L$ by $\p$ can be uniquely split into distinct maximal ideals $\pp_j$ lying over $\p$ in the following way: there exists an integer $g\geq1$ such that
\begin{equation*}
\p\O_L=\prod_{j=1}^g\pp_j^{e_j}.
\end{equation*}
We say the ideal $\p$ is unramified in $L$ if $e_j=1$ for all $1\leq j\leq g$, which occurs for all but finitely many prime ideals.  Define the absolute norm of a nonzero ideal $\mathfrak{a}$ of the ring of integers $\O_F$ of some number field $F$ by
\begin{equation*}
\mathrm{Nm}(\mathfrak{a}):=\left[\O_F:\mathfrak{a}\right]=\left|\O_F/\mathfrak{a}\right|.
\end{equation*}
For any prime ideal $\pp$ lying over $\p$, the Artin symbol $\left[\frac{L/K}{\pp}\right]$ is defined as the unique element $\sigma\in G$ such that
\begin{equation*}
\sigma(\a)=\a^{\mathrm{Nm}(\p)}\pmod{\pp}\text{ for all }\a\in L.
\end{equation*}
All of the prime ideals $\pp_j$ lying over $\p$ are isomorphic by elements of $G$ and $$\left[\frac{L/K}{\t(\pp)}\right]=\t\left[\frac{L/K}{\pp}\right]\t^{-1}$$ for $\t\in G$, so there exists a conjugacy class $C$ associated to $\p$ such that each $\left[\frac{L/K}{\pp_j}\right]$ lies in $C$.  We define the Artin symbol $\left[\frac{L/K}{\p}\right]$ to be the conjugacy class $C$.\\

We now define density.  Let $K$ be a number field, let $Q(K)$ be some set of prime ideals of $\O_K$, and let $P(K)$ be the set of all prime ideals of $\O_K$.  The natural density of $Q(K)$ is defined by
\begin{equation*}
\lim_{X\rightarrow\infty}\frac{\#\{\p\in Q(K):\mathrm{Nm}(\p)\leq X\}}{\#\{\p\in P(K):\mathrm{Nm}(\p)\leq X\}},
\end{equation*}
provided the limit exists.

For convenience, define
\begin{eqnarray*}
\pi_C(X,L/K)&:=&\#\{\p\in P_C:\mathrm{Nm}(\p)\leq X\}.
\end{eqnarray*}
In other words, $\pi_C(X,L/K)$ is the number of nonzero prime ideals $\p$ of $\O_K$ which are unramified in $L$ and for which $\mathrm{Nm}(\p)\leq X$ and $\left[\frac{L/K}{\p}\right]=C$, where $C$ is a conjugacy class of the Galois group $G=\mathrm{Gal}(L/K)$.

We may now recall the Chebotarev Density Theorem \cite{C}.  Let $L/K$ be a finite Galois extension of number fields, and let $C$ be a conjugacy class of $\mathrm{Gal}(L/K)$.  Let
\begin{equation*}
P_C=\left\{\p\in P(K):\p\text{ is unramified in }L,\left[\frac{L/K}{\p}\right]=C\right\}.
\end{equation*}
Then, as $X\rightarrow\infty$, we have that
\begin{equation*}
\pi_C(X,L/K)=\frac{\#C}{\#\mathrm{Gal}(L/K)}\cdot\frac{X}{\log X}+o\left(\frac{X}{\log X}\right).
\end{equation*}
In other words, the natural density of $P_C$ in $\{\p\in P(K):\p\text{ is unramified in }L\}$ exists and is equal to $\frac{\#C}{\#\mathrm{Gal}(L/K)}$.\\

In particular, a more precise formulation of the Chebotarev Density Theorem from Lagarias and Odlyzko \cite[Theorems 1.3 and 1.4]{S} is:

\begin{thm}[Lagarias--Odlyzko \cite{S}]
For sufficiently large $X\geq c_1\left(D_L,n_L\right)$, where the constant $c_1$ depends on both the discriminant $D_L$ and the degree $n_L$ of $L$, we have
\begin{equation*}
\left|\pi_C(X,L/K)-\frac{\#C}{\#G}\,\Li(X)\right|\leq2c_2\,X\,\mathrm{exp}\left\{-c_3\left(n_L\right)^{-1/2}\sqrt{\log X}\right\}
\end{equation*}
for constants $c_2$ and $c_3$.
\end{thm}

Note that $\Li(X):=\int_2^Xdt/\log t$.

\begin{rmk}
The constant $c_1$ can be made explicit using Theorems 1.3 and 1.4 of \cite{S}.
\end{rmk}

The above theorem is useful in bounding error terms in the proofs of Theorems \ref{main} and \ref{lpd}.

\vspace{.5cm}

\section{Proofs}\label{3}

The proofs of Theorems \ref{main} and \ref{lpd} closely follow the proofs of the corresponding theorems in \cite{A}.  Note that $c_4,\dots,c_{12}$ are positive constants which will not be specified in the following proofs.

\begin{proof}[Proof of Theorem \ref{lpd}]
We first rewrite the desired sum in terms of the function $\Psi(X,Y)$, for which we have asymptotic bounds.
\begin{eqnarray*}
\sum_{\substack{2\leq n\leq X\\[1pt]\left[\frac{K/\Q}{p_{\mathrm{max}}(n)}\right]=C}}1&=&\sum_{\substack{p\leq X\\[1pt]\left[\frac{K/\Q}{p}\right]=C}}\,\,\sum_{\substack{n\leq X\\[1pt]p_{\mathrm{max}}(n)=p}}1\\
&=&\sum_{\substack{p\leq X\\[1pt]\left[\frac{K/\Q}{p}\right]=C}}\Psi\left(\frac{X}{p},p\right).
\end{eqnarray*}
Notice that this sum can be broken up into a sum over small primes and a sum over large primes, so that
\begin{equation*}
\sum_{\substack{2\leq n\leq X\\[1pt]\left[\frac{K/\Q}{p_{\mathrm{max}}(n)}\right]=C}}1=\sum_{\substack{p\leq\mathrm{exp}\left\{(\log X)^{2/3}\right\}\\[1pt]\left[\frac{K/\Q}{p}\right]=C}}\Psi\left(\frac{X}{p},p\right)+\sum_{\substack{\mathrm{exp}\left\{(\log X)^{2/3}\right\}<p\leq X\\[1pt]\left[\frac{K/\Q}{p}\right]=C}}\Psi\left(\frac{X}{p},p\right).
\end{equation*}
Let
\begin{equation*}
S_1:=\sum_{\substack{p\leq\mathrm{exp}\left\{(\log X)^{2/3}\right\}\\[1pt]\left[\frac{K/\Q}{p}\right]=C}}\Psi\left(\frac{X}{p},p\right)\hspace{.5cm}\text{and}\hspace{.5cm}S_2:=\sum_{\substack{\mathrm{exp}\left\{(\log X)^{2/3}\right\}<p\leq X\\[1pt]\left[\frac{K/\Q}{p}\right]=C}}\Psi\left(\frac{X}{p},p\right).
\end{equation*}
We now estimate $S_1$ and show that it is much smaller than $S_2$.  This implies that $S_1$ is not the main term in the asymptotic formula in equation (\ref{lpdeq}), so we only need to obtain an upper bound.  We see that
\begin{eqnarray*}
S_1&=&\sum_{\substack{p\leq\mathrm{exp}\left\{(\log X)^{2/3}\right\}\\[1pt]\left[\frac{K/\Q}{p}\right]=C}}\Psi\left(\frac{X}{p},p\right)\\
&\leq&\sum_{p\leq\mathrm{exp}\left\{(\log X)^{2/3}\right\}}\Psi\left(\frac{X}{p},p\right).
\end{eqnarray*}
Let
\begin{equation*}
Y=\mathrm{exp}\left\{(\log X)^{2/3}\right\}.
\end{equation*}
Then we have that
\begin{equation*}
S_1\leq\Psi(X,Y)-1.
\end{equation*}
If $Y=X^{1/\b}$ for some $\b$, then $\b=(\log X)^{1/3}$.  Thus, by equation \eqref{dickman}, we have that
\begin{equation*}
S_1=O\left(X\,\mathrm{exp}\left\{-(\log X)^{1/3}\log\log X\right\}\right).
\end{equation*}
We will now estimate $S_2$, and it turns out that this will provide the main term in the asymptotic formula in equation (\ref{lpdeq}).  To obtain the main term, it will be convenient to define
\begin{equation*}
S_3:=\sum_{\substack{\mathrm{exp}\left\{(\log X)^{2/3}\right\}<p\leq X\\[1pt]\left[\frac{K/\Q}{p}\right]=C}}\Psi\left(\frac{X}{p},p\right)-\frac{\#C}{\#G}\int_{\mathrm{exp}\left\{(\log X)^{2/3}\right\}}^X\Psi\left(\frac{X}{t},t\right)\frac{dt}{\log t},
\end{equation*}
which means that
\begin{equation*}
S_2=\frac{\#C}{\#G}\int_{\mathrm{exp}\left\{(\log X)^{2/3}\right\}}^X\Psi\left(\frac{X}{t},t\right)\frac{dt}{\log t}-S_3.
\end{equation*}
Our goal now is to show that $S_3$ is small compared to $S_2$.  By the definition of $\Psi\left(\frac{X}{t},t\right)$, we replace it with a function counting elements of $\mathcal{S}\left(\frac{X}{t},t\right)$ to obtain
\begin{equation*}
S_3=\sum_{\substack{\mathrm{exp}\left\{(\log X)^{2/3}\right\}<p\leq X\\[1pt]\left[\frac{K/\Q}{p}\right]=C}}\left(\sum_{n\in \mathcal{S}(X/p,p)}1\right)-\frac{\#C}{\#G}\int_{\mathrm{exp}\left\{(\log X)^{2/3}\right\}}^X\left(\sum_{n\in \mathcal{S}(X/t,t)}1\right)\frac{dt}{\log t}.
\end{equation*}
Applying Lemma \ref{tetris} and switching the order of summation in the first term, we then have that
\begin{eqnarray*}
S_3&=&\sum_{\substack{1\leq n\leq X\,\mathrm{exp}\left\{-(\log X)^{2/3}\right\}\\[1pt]p_{\mathrm{max}}(n)\leq X/n}}\left(\sum_{\substack{p_{\mathrm{max}}(n)\leq p\leq X/n\\[1pt]p>\mathrm{exp}\left\{(\log X)^{2/3}\right\}\\[1pt]\left[\frac{K/\Q}{p}\right]=C}}1-\frac{\#C}{\#G}\int_{\max\left(p_{\mathrm{max}}(n),\mathrm{exp}\{(\log X)^{2/3}\}\right)}^{X/n}\frac{dt}{\log t}\right)\\
&=&\sum_{\substack{1\leq n\leq X\,\mathrm{exp}\left\{-(\log X)^{2/3}\right\}\\[1pt]p_{\mathrm{max}}(n)\leq X/n}}\Bigg(\pi_C\left(\frac{X}{n},K/\Q\right)-\pi_C\left(\max\left(p_{\mathrm{max}}(n),\mathrm{exp}\left\{(\log X)^{2/3}\right\}\right),K/\Q\right)\\
&&\hspace{3.5cm}-\frac{\#C}{\#G}\,\Li(X/n)+\frac{\#C}{\#G}\,\Li\left(\max\left(p_{\mathrm{max}}(n),\mathrm{exp}\left\{(\log X)^{2/3}\right\}\right)\right)\Bigg),
\end{eqnarray*}
by the definitions of $\pi_C(X,K/\Q)$ and $\Li(X)$.  Here we apply the reformulation of the Chebotarev Density Theorem by Lagarias and Odlyzko to obtain
\begin{equation*}
\left|S_3\right|\leq\sum_{\substack{1\leq n\leq X\,\mathrm{exp}\left\{-(\log X)^{2/3}\right\}\\[1pt]p_{\mathrm{max}}(n)\leq X/n}}c_4\,(X/n)\,\mathrm{exp}\left\{-c_5\sqrt{\log(X/n)}\right\}.
\end{equation*}
Since each summand satisfies
\begin{equation*}
c_4\,(X/n)\,\mathrm{exp}\left\{-c_5\sqrt{\log(X/n)}\right\}\leq c_4\,(X/n)\,\mathrm{exp}\left\{-c_6(\log X)^{1/3}\right\},
\end{equation*}
we have that
\begin{equation*}
S_3=O\left(X\,\exp\left\{-c_7(\log X)^{1/3}\right\}\right).
\end{equation*}
We have used the fact that an absolute value upper bound of the remainder term in the Chebotarev Density Theorem is an increasing function of $X$, so we have replaced the terms $p_{\mathrm{max}}(n)$ and $\mathrm{exp}\left\{(\log X)^{2/3}\right\}$ by $X/n$.  In order to get the main term of the asymptotic formula from $S_2$, we must show that the integral
\begin{equation*}
\int_{\mathrm{exp}\left\{(\log X)^{2/3}\right\}}^X\Psi\left(\frac{X}{t},t\right)\frac{dt}{\log t}
\end{equation*}
contributes a factor of $X$.  Let $[X]$ denote the integral part of $X$, and $\{X\}$ the fractional part, so that $X=[X]+\{X\}$.  We observe that
\begin{eqnarray*}
[X]-1&=&\sum_{2\leq n\leq X}1\\
&=&\sum_{p\leq X}\Psi\left(\frac{X}{p},p\right),
\end{eqnarray*}
which we again break up into small primes and large primes so that
\begin{eqnarray*}
[X]-1&=&\sum_{p\leq\exp\left\{(\log X)^{2/3}\right\}}\Psi\left(\frac{X}{p},p\right)+\sum_{\exp\left\{(\log X)^{2/3}\right\}<p\leq X}\Psi\left(\frac{X}{p},p\right).
\end{eqnarray*}
Let
\begin{equation*}
{S_1}':=\sum_{p\leq\exp\left\{(\log X)^{2/3}\right\}}\Psi\left(\frac{X}{p},p\right)\hspace{.5cm}\text{and}\hspace{.5cm}{S_2}':=\sum_{\exp\left\{(\log X)^{2/3}\right\}<p\leq X}\Psi\left(\frac{X}{p},p\right).
\end{equation*}
By similar estimates, we have that
\begin{equation*}
{S_1}'=O\left(X\,\exp\left\{-(\log X)^{1/3}(\log\log X)\right\}\right)
\end{equation*}
and
\begin{equation*}
{S_2}'=\int_{\exp\left\{(\log X)^{2/3}\right\}}^X\Psi\left(\frac{X}{t},t\right)\frac{dt}{\log t}+{S_3}',
\end{equation*}
where
\begin{equation*}
{S_3}'=O\left(X\,\exp\left\{-c_8(\log X)^{1/3}\right\}\right).
\end{equation*}
Combining these estimates gives
\begin{equation*}
\int_{\exp\left\{(\log X)^{2/3}\right\}}^X\Psi\left(\frac{X}{t},t\right)\frac{dt}{\log t}=[X]+O({S_1}'+{S_3}').
\end{equation*}
Thus we obtain the desired asymptotic formula,
\begin{eqnarray*}
\sum_{\substack{2\leq n\leq X\\[1pt]\left[\frac{K/\Q}{p_{\mathrm{max}}(n)}\right]=C}}1&=&\frac{\#C}{\#G}\cdot X+O\left(S_1+S_3+{S_1}'+{S_3}'\right)\\
&=&\frac{\#C}{\#G}\cdot X+O\left(X\,\mathrm{exp}\left\{-c_9(\log X)^{1/3}\right\}\right).
\end{eqnarray*}
\end{proof}

As a consequence of Theorem \ref{lpd}, we have the following lemma.
\begin{Lemma}\label{int}
Assume the notation and hypotheses from Theorem \ref{main}.  Then we have that
\begin{equation*}
\sum_{\substack{2\leq n\leq X \\[1pt] \left[\frac{K/\Q}{p_{\mathrm{max}}(n)}\right]=C}}
\frac{1}{n}=\frac{\#C}{\#G}\cdot\log X+O\left(\mathrm{exp}\left\{-k(\log X)^{1/3}\right\}\right)
\end{equation*}
where $k$ is a positive constant.
\end{Lemma}

\begin{proof}[Proof of Lemma \ref{int}]
Define the function $f$ by
\begin{equation}\label{f}
f(n):=\left\{
		\begin{array}{ll}
			1 & \mbox{if } \left[\frac{K/\Q}{p}\right]=C,\hspace{.3cm}n=p>1, \\
			0 & \mbox{otherwise, }
		\end{array}
	\right.
\end{equation}
and set
\begin{equation*}
\psi_f(X):=\sum_{n\leq X}f\big(p_{\mathrm{max}}(n)\big).
\end{equation*}
Then $\psi_f(X)$ counts the number of integers $n\leq X$ such that $\left[\frac{K/\Q}{p_{\mathrm{max}}(n)}\right]=C$, so by Theorem \ref{lpd}  we have that
\begin{equation*}
\psi_f(X)=\frac{\#C}{\#G}\cdot X+e_f(X),
\end{equation*}
where $e_f(X)=O\left(X\,\mathrm{exp}\left\{-k(\log X)^{1/3}\right\}\right)$.  The function $\psi_f(X)$ is a type of ``stair-step function," meaning it oscillates as (the integral part of) $X$ increases depending on the values of $p_{\mathrm{max}}(n)$ for $n\leq X$.  Then we can rewrite
\begin{equation*}
\sum_{\substack{2\leq n\leq X\\[1pt]\left[\frac{K/\Q}{p_{\mathrm{max}}(n)}\right]=C}}\frac{1}{n}=\int_1^X\frac{d\psi_f(t)}{t},
\end{equation*}
which by Theorem \ref{lpd} is
\begin{eqnarray*}
\int_1^X\frac{d\psi_f(t)}{t}&=&\frac{\#C}{\#G}\int_1^X\frac{dt}{t}+\int_1^X\frac{de_f(t)}{t}\\
&=&\frac{\#C}{\#G}\cdot\log X+\frac{e_f(t)}{t}\bigg\vert_1^X+\int_1^X\frac{e_f(t)\,dt}{t^2}\\
&=&\frac{\#C}{\#G}\cdot\log X+c_1-\int_X^\infty\frac{e_f(t)\,dt}{t^2}+\frac{e_f(X)}{X},
\end{eqnarray*}
where
\begin{equation*}
c_1=\frac{-e_f(1)}{1}+\int_1^\infty\frac{e_f(t)\,dt}{t^2}.
\end{equation*}
Note that the number $c_1$ exists by Theorem \ref{lpd}, and that Lemma \ref{int} now follows because
\begin{equation*}
\frac{e_f(X)}{X}=O\left(\mathrm{exp}\left\{-k(\log X)^{1/3}\right\}\right)
\end{equation*}
where $k$ is a positive constant.
\end{proof}

\begin{proof}[Proof of Theorem \ref{main}]
Let $f$ be defined as in (\ref{f}) above.  By equation (\ref{Alladi1}) from the duality theorem of Alladi, we have
\begin{eqnarray*}
\sum_{\substack{n\leq X\\[1pt]\left[\frac{K/\Q}{p_{\mathrm{min}}(n)}\right]=C}}\frac{\m(n)}{n}&=&\sum_{n\leq X}\frac{\m(n)f(p_{\mathrm{min}}(n))}{n}\\
&=&-\sum_{n\leq X}\frac{1}{n}\sum_{d\vert n}\m\left(\frac{n}{d}\right)f(p_{\mathrm{max}}(d))\\
&=&-\sum_{n\leq X}\sum_{d\vert n}\frac{\m(n/d)}{n/d}\cdot\frac{f(p_{\mathrm{max}}(d))}{d}.
\end{eqnarray*}
To more easily obtain estimates, we delicately split the double sum into two double sums by introducing the variable $m:=n/d$.  For each such $m$, the allowed values of $d$ with $dm=n<X$ are exactly $1\leq d\leq X/m$, so we have that
\begin{eqnarray*}
&&-\sum_{n\leq X}\sum_{d\vert n}\frac{\m(n/d)}{n/d}\cdot\frac{f(p_{\mathrm{max}}(d))}{d}\\
&&\hspace{2cm}=-\sum_{1\leq m\leq\sqrt{X}}\frac{\m(m)}{m}\sum_{d\leq X/m}\frac{f(p_{\mathrm{max}}(d))}{d}-\sum_{\sqrt{X}<m\leq X}\frac{\m(m)}{m}\sum_{d\leq X/m}\frac{f(p_{\mathrm{max}}(d))}{d}.
\end{eqnarray*}
Now we change the order of summation in the second sum.  We use the fact that $m>\sqrt{X}$ and $md=n\leq X$ implies $d<\sqrt{X}$ to obtain
\begin{eqnarray*}
&&-\sum_{n\leq X}\sum_{d\vert n}\frac{\m(n/d)}{n/d}\cdot\frac{f(p_{\mathrm{max}}(d))}{d}\\
&&\hspace{2cm}=-\sum_{1\leq m\leq\sqrt{X}}\frac{\m(m)}{m}\sum_{d\leq X/m}\frac{f(p_{\mathrm{max}}(d))}{d}-\sum_{d<\sqrt{X}}\frac{f(p_{\mathrm{max}}(d))}{d}\sum_{\sqrt{X}<m\leq X/d}\frac{\m(m)}{m}.
\end{eqnarray*}
We estimate the two sums separately, reverting back to the variable $n$ instead of the new variable $m$.  Let
\begin{equation*}
S_6:=-\sum_{n\leq\sqrt{X}}\frac{\m(n)}{n}\sum_{d\leq X/n}\frac{f(p_{\mathrm{max}}(d))}{d}\hspace{.5cm}\text{and}\hspace{.5cm}S_7:=-\sum_{n<\sqrt{X}}\frac{f(p_{\mathrm{max}}(n))}{n}\sum_{\sqrt{X}<d\leq X/n}\frac{\m(d)}{d}.
\end{equation*}
We will show that $S_6$ gives the main term of the desired asymptotic formula, and we will bound $S_7$.  By Lemma \ref{int}, we have that
\begin{eqnarray*}
S_6&=&-\sum_{n\leq\sqrt{X}}\frac{\m(n)}{n}\left[\frac{\#C}{\#G}\cdot\log\left(\frac{X}{n}\right)+O\left(\mathrm{exp}\left\{-k(\log(X/n))^{1/3}\right\}\right)\right]\\
&=&-\left(\frac{\#C}{\#G}\cdot\log X\right)\sum_{n\leq\sqrt{X}}\frac{\m(n)}{n}+\frac{\#C}{\#G}\sum_{1\leq n\leq\sqrt{X}}\frac{\m(n)\,\log n}{n}+O\left(\mathrm{exp}\left\{-k(\log X)^{1/3}\right\}\right).
\end{eqnarray*}
We now apply well-known bounds for
\begin{equation*}
\sum_{n\leq\sqrt{X}}\frac{\m(n)}{n}\hspace{.5cm}\text{and}\hspace{.5cm}\sum_{1\leq n\leq\sqrt{X}}\frac{\m(n)\,\log n}{n}
\end{equation*}
which are consequences of the standard zero-free region for $\zeta(s)$ (for example, see \cite[Chapter 13]{D}).  Namely, we have that
\begin{equation}\label{mu}
\sum_{n\leq\sqrt{X}}\frac{\m(n)}{n}=O\left(\exp\left\{-c_{10}(\log X)^{1/2}\right\}\right)
\end{equation}
and
\begin{equation*}
\sum_{1\leq n\leq\sqrt{X}}\frac{\m(n)\,\log n}{n}=-1+O\left(\exp\left\{-c_{11}(\log X)^{1/2}\right\}\right).
\end{equation*}
Therefore, we have that
\begin{eqnarray*}
S_6&=&-\frac{\#C}{\#G}+O\left(\exp\left\{-c_{10}(\log X)^{1/2}\right\}\right)+O\left(\exp\left\{-c_{11}(\log X)^{1/2}\right\}\right)\\
&&\hspace{1.5cm}+O\left(\mathrm{exp}\left\{-k(\log X)^{1/3}\right\}\right)\\
&=&-\frac{\#C}{\#G}+O\left(\mathrm{exp}\left\{-k(\log X)^{1/3}\right\}\right).
\end{eqnarray*}
By equation (\ref{mu}), we also have that
\begin{eqnarray*}
S_7&=&O\left(\sum_{n\leq\sqrt{X}}\frac{1}{n}\exp\left\{-c_{10}(\log(X/n))^{1/2}\right\}\right)\\
&=&O\left(\exp\left\{-c_{12}(\log X)^{1/2}\right\}\right).
\end{eqnarray*}
Now we see that
\begin{equation*}
\sum_{\substack{2\leq n\leq X\\[1pt]\left[\frac{K/\Q}{p_{\mathrm{min}}(n)}\right]=C}}\frac{\m(n)}{n}=-\frac{\#C}{\#G}+S_6+S_7,
\end{equation*}
and therefore we have that
\begin{equation}\label{convrate}
\sum_{\substack{2\leq n\leq X\\[1pt]\left[\frac{K/\Q}{p_{\mathrm{min}}(n)}\right]=C}}\frac{\m(n)}{n}=-\frac{\#C}{\#G}+O\left(\mathrm{exp}\left\{-k(\log X)^{1/3}\right\}\right).
\end{equation}
Note that as $X\rightarrow\infty$, the error term $1/\mathrm{exp}\left\{k(\log X)^{1/3}\right\}\rightarrow0$.  Thus we conclude that
\begin{equation*}
\lim_{X\rightarrow\infty}\sum_{\substack{2\leq n\leq X\\[1pt]\left[\frac{K/\Q}{p_{\mathrm{min}}(n)}\right]=C}}\frac{\m(n)}{n}=-\frac{\#C}{\#G}.
\end{equation*}
\end{proof}

\end{document}